\documentclass[12pt]{article}
\usepackage[all]{xy}
\usepackage[latin5]{inputenc}
\usepackage{rotating}
\usepackage{amsfonts}
\usepackage{amssymb}
\usepackage{amsmath}
\usepackage{cite}
\usepackage{bbm}
\usepackage{graphicx}
\usepackage{pstricks,pst-math,pst-infixplot}
\usepackage[dotinlabels]{titletoc}
\usepackage{psfrag}

\newcounter{theorem}
\newtheorem{theorem}{Theorem}

\newtheorem{corollary}{Corollary}
\newtheorem{definition}{Definition}

\newtheorem{lemma}{Lemma}
\newtheorem{proposition}{Proposition}
\newtheorem{remark}{Remark}
\newtheorem{question}{Question}
\newenvironment{proof}[1][Proof]{\textbf{#1.} }{\rule{0.5em}{0.5em}}

\usepackage[a4paper,left=2.5cm,right=2.5cm,top=2cm,bottom=2.5cm]{geometry}

\begin{document}

\title{Isometry Classes of Planes in $(\mathbb{R}^3,d_{\infty})$}
\author{Mehmet KILI\c{C}}
\date{ }
\maketitle

\begin{abstract}
We determine geodesics in $\mathbb{R}_{\infty}^n$ (i.e. $(\mathbb{R}^n,d_{\infty})$) and by using this, classify planes up to isometry in $\mathbb{R}_{\infty}^3$.
\end{abstract}



\section{Introduction}

In metric spaces, it is possible to define length of paths. Let $(X,d)$ be a metric space and $\alpha:[0,1]\rightarrow X$ be a path. Then, the length of $\alpha$ is defined as $$\sup_\mathcal{P}\left\{\sum_{i=1}^{n}d(\alpha(t_{i-1}),\alpha(t_i))\right\}$$ over all partitions $P=\{t_0=0,t_1,\ldots,t_n=1\}$ of $[0,1]$ and it is denoted by $L(\alpha)$. If $\alpha$ satisfies $L(\alpha|_{[0,t]})=t\cdot L(\alpha)$ for all $t\in(0,1)$, then $\alpha$ is called a natural path. It is clear that every path has a natural reparametrization. If the path $\alpha$ is natural and satisfies $L(\alpha)=d(x,y)$, where $\alpha(0)=x$, $\alpha(1)=y$, then $\alpha$ is called a geodesic. For a metric space and any two points in it, there may not exist any geodesic between these points. For example, $S^1=\{(x,y)\in\mathbb{R}^2\,|\, x^2+y^2=1\}$ with induced standard metric and for $a=(1,0)$ and $b=(-1,0)$, there is no path connecting these points whose length is less than $\pi$, but the distance between $a$ and $b$ is equal to $2$ according to the standard metric. If there is at least one geodesic between any two points in a metric space,  this metric space is called ''geodesic space'' (\cite{bri}, \cite{pap}), or ''strictly intrinsic space'' according to another terminology (\cite{bur}). So, $S^1$ is not a geodesic space with the induced standard metric, but it becomes a geodesic space with the ''arc length metric''.  $(\mathbb{R}^n,d_p)$ with
\[
d_p((x_1,x_2,\ldots,x_n),(y_1,y_2\ldots,y_n))=\sqrt[p]{|x_1-y_1|^p+|x_2-y_2|^p+\cdots+|x_n-y_n|^p}
\]
for $1\leq p<\infty$ and $(\mathbb{R}^n,d_{\infty})$ with
\[
d_{\infty}((x_1,x_2,\ldots,x_n),(y_1,y_2\ldots,y_n))=\max_{i=1}^n\{|x_i-y_i|\}
\]
are also geodesic spaces.  The suitably parameterized segment between the points $(x_1,x_2,\ldots,x_n)$ and $(y_1,y_2\ldots,y_n)$ is a geodesic in $(\mathbb{R}^n,d_p)$ for $1\leq p\leq\infty$.

Let $(X,d_X)$ and $(Y,d_Y)$ be two metric spaces and $f$ be a function from $X$ onto $Y$. $f$ is called an isometry if $d_Y(f(x),f(y))=d_X(x,y)$ for all $x,y\in X$. If there is an isometry between two metric spaces, then they are called isometric spaces. It is clear that notion of isometry deals with not only set on the space but also metric on the space. For example, $S^1$ with the induced standard metric from $\mathbb{R}^2$ and the same set with the ''arc length metric'' are not isometric.

In our previous work (\cite{kil}), we have noted that the plane $$L=\{(x,y,z)\,|\ x+y+z=0\}\subseteq\mathbb{R}_{\infty}^3$$ with the induced metric is not hyperconvex; therefore, it is not isometric to the plane $\mathbb{R}_{\infty}^2$ because $\mathbb{R}_{\infty}^2$ is hyperconvex. But it is clear that the $xy-$plane in $\mathbb{R}_{\infty}^3$ is isometric to the plane $\mathbb{R}_{\infty}^2$; hence, all planes in $\mathbb{R}_{\infty}^3$ are not isometric to each other. Therefore, the following question arises:
\begin{question}
What are the isometry classes of planes in $\mathbb{R}_{\infty}^3$?
\end{question}
In this paper, we have answered this question without using the notion of hyperconvexity. In order to do that, first, we have determined geodesics in $\mathbb{R}_{\infty}^n$ and then we have achieved our main aim.

\section{Geodesics in $\mathbb{R}_{\infty}^n$}

\begin{definition}
For $p=(p_1,p_2,\ldots,p_n)\in \mathbb{R}^n_{\infty}$, we define
\[
S_i^{\varepsilon}(p)=\{q=(q_1,q_2,\ldots,q_n)\in\mathbb{R}^n|\ d_{\infty}(p,q)=\varepsilon(q_i-p_i)\}
\]
for $i=1,2,\ldots,n$ and
$\varepsilon=\pm$ and call them the sectors at the point $p$ (see
Fig.~\ref{fd1} and Fig.~\ref{fd2}).
\end{definition}

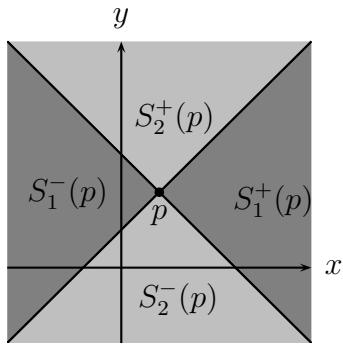
\begin{figure}[h]
\begin{center}
\begin{pspicture*}(-1.5,-1)(3,3.5)
\pspolygon*[linecolor=gray](0.5,1)(2.5,3)(2.5,-1)
\psline(0.5,1)(2.5,3) \psline(0.5,1)(2.5,-1)
\pspolygon*[linecolor=lightgray](0.5,1)(2.5,3)(-1.5,3)
\psline(0.5,1)(-1.5,3) \psline(0.5,1)(2.5,3)
\pspolygon*[linecolor=gray](0.5,1)(-1.5,-1)(-1.5,3)
\psline(0.5,1)(-1.5,-1) \psline(0.5,1)(-1.5,3)
\pspolygon*[linecolor=lightgray](0.5,1)(-1.5,-1)(2.5,-1)
\psline(0.5,1)(2.5,-1) \psline(0.5,1)(-1.5,-1)
\psline{->}(-1.5,0)(2.5,0) \uput[r](2.5,0){$x$}
\psline{->}(0,-1.5)(0,3) \uput[u](0,3){$y$} \psdot(0.5,1)
\uput[d](0.5,1){$p$} \uput[d](2,1.3){$S_1^+(p)$}
\uput[u](0.7,1.6){$S_2^+(p)$} \uput[l](0,1){$S_1^-(p)$}
\uput[d](0.75,0){$S_2^-(p)$}
\end{pspicture*}
\caption{Sectors of a point $p$ in $\mathbb{R}^2_{\infty}$.}
\label{fd1}
\end{center}
\end{figure}

\begin{theorem}
\label{d1}
Let $p=(p_1,p_2,\ldots,p_n)$, $q=(q_1,q_2,\ldots,q_n)\in\mathbb{R}^n$ be two points, $q\in S_i^\varepsilon(p)$ and $\alpha:[0,1]\rightarrow \mathbb{R}^n$ be a natural path such that $\alpha(0)=p$ and $\alpha(1)=q$. Then $\alpha$ is a geodesic in $\mathbb{R}_\infty^n$ if and only if $\alpha(t')\in S_i^\varepsilon(\alpha(t))$ for all $t,t'\in [0,1]$ such that $t<t'$.
\end{theorem}

\begin{proof}
($\Rightarrow$)

Let $\alpha=(\alpha_1,\alpha_2,\ldots,\alpha_n)$ be a geodesic and assume that $\alpha(t')\notin S_i^\varepsilon(\alpha(t))$ for some $t<t'$. Then, we have $d_\infty(\alpha(t),\alpha(t'))>\varepsilon(\alpha_i(t')-\alpha_i(t))$. So if we take the partition $0<t<t'<1$ of $[0,1]$, we have
\begin{eqnarray*}
d_\infty(\alpha(0),\alpha(t))+d_\infty(\alpha(t),\alpha(t'))+d_\infty(\alpha(t'),\alpha(1))&>&\\
\varepsilon(\alpha_i(t)-\alpha_i(0))+\varepsilon(\alpha_i(t')-\alpha_i(t))+\varepsilon(\alpha_i(1)-\alpha_i(t'))&=&\varepsilon(\alpha_i(1)-\alpha_i(0))\\
&=&\varepsilon(q_i-p_i)\\
&=&d_\infty(p,q).
\end{eqnarray*}
This leads to  the contradiction that $L(\alpha)>d_\infty(p,q)$.

($\Leftarrow$)
Let $0=t_0<t_1<\cdots<t_n=1$ be an arbitrary partition of $[0,1]$. Since $\alpha(t_j)\in S_i^\varepsilon(\alpha(t_{j-1}))$ for all $j=1,2,\ldots,n$, we have $d_\infty(\alpha(t_j),\alpha(t_{j-1}))=\varepsilon(\alpha_i(t_j)-\alpha_i(t_{j-1}))$; so,
\begin{eqnarray*}
\sum_{j=1}^nd_\infty(\alpha(t_j),\alpha(t_{j-1}))&=&\sum_{j=1}^n\varepsilon(\alpha_i(t_j)-\alpha_i(t_{j-1})) \\ &=&
\varepsilon(\alpha_i(1)-\alpha_i(0))\\&=&\varepsilon(q_i-p_i)\\&=&d_{\infty}(p,q).
\end{eqnarray*}
This implies that $L(\alpha)=d_{\infty}(p,q)$ and that $\alpha$ is a geodesic.
\end{proof}

\begin{figure}[h!]
\begin{center}
\begin{pspicture*}(-2,-0.5)(8,6.5)
\psset{unit=0.85}
\pspolygon*[linecolor=lightgray](0,4)(3.25,2.75)(6,6)(2,6)(0,4)
\psline(0,4)(4,4)(4,0)
\psline(0,4)(2,6)(6,6)(6,2)(4,0)
\psline(6,6)(4,4)
\psline(4,0)(0,0)(0,4)
\psline[linewidth=0.5pt, linestyle=dashed](0,0)(2,2)(6,2)
\psline[linewidth=0.5pt, linestyle=dashed](2,2)(2,6)
\psdot(3.25,2.75)
\uput[d](3.25,2.75){$O$}
\psline(3.25,2.75)(2,6)
\psline(3.25,2.75)(6,6)
\psline(3.25,2.75)(0,4)
\psline(3.25,2.75)(4,4)
\uput[u](3.25,4.5){$S_3^{+}(O)$}
\psline{->}(3.25,2.75)(2.9,2.4)
\uput[d](2.75,2.75){$x$}
\psline{->}(3.25,2.75)(3.75,2.75)
\uput[r](3.45,2.88){$y$}
\psline{->}(3.25,2.75)(3.25,3.25)
\uput[u](3.25,3.25){$z$}
\end{pspicture*}
\caption{The sector $S_3^{+}(O)$ of the origin in $\mathbb{R}^3_{\infty}$.}
\label{fd2}
\end{center}
\end{figure}
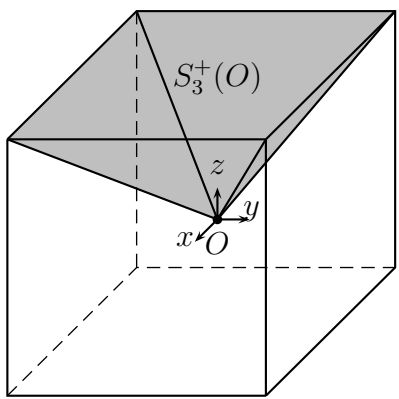

Theorem~\ref{d1} can be restated as follows: Let $\varepsilon$ and $i$ be such that $q\in S_i^{\varepsilon}(p)$, then the natural path $\alpha$ between $p$ and $q$ is a geodesic if and only if when the sectors $S_i^{\varepsilon}(.)$ travel on the image of $\alpha$, the rest of the path is contained from the sector at every point (see Figure~\ref{fd3}).

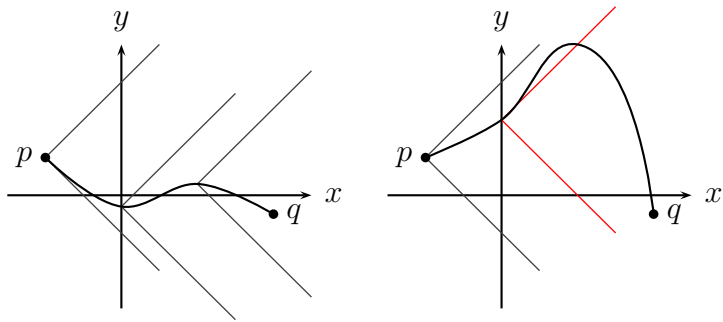
\begin{figure}[h]
\begin{center}
\begin{pspicture*}(-1.5,-2)(8,2.5)
\psline{->}(-1.5,0)(2.5,0) \uput[r](2.5,0){$x$}
\psline{->}(0,-1.5)(0,2) \uput[u](0,2){$y$}
\psline[linewidth=0.5pt, linecolor=darkgray](0.5,2)(-1,0.5)(0.5,-1)
\psline[linewidth=0.5pt, linecolor=darkgray](1.5,1.35)(0,-0.15)(1.5,-1.65)
\psline[linewidth=0.5pt, linecolor=darkgray](2.5,1.65)(1,0.15)(2.5,-1.35)
\psdot(-1,0.5)
\uput[l](-1,0.5){$p$}
\psdot(2,-0.25)
\uput[r](2,-0.25){$q$}
\pscurve(-1,0.5)(0,-0.15)(1,0.15)(2,-0.25)

\psline{->}(3.5,0)(7.5,0) \uput[r](7.5,0){$x$}
\psline{->}(5,-1.5)(5,2) \uput[u](5,2){$y$}
\psline[linewidth=0.5pt, linecolor=darkgray](5.5,2)(4,0.5)(5.5,-1)
\psline[linewidth=0.5pt, linecolor=red](6.5,2.5)(5,1)(6.5,-0.5)
\psdot(4,0.5)
\uput[l](4,0.5){$p$}
\psdot(7,-0.25)
\uput[r](7,-0.25){$q$}
\pscurve(4,0.5)(5,1)(6,2)(7,-0.25)
\end{pspicture*}
\caption{Two paths between $p$ and $q$ in $\mathbb{R}^2_{\infty}$ one of which (on the left) is a geodesic but the other is not.}
\label{fd3}
\end{center}
\end{figure}

Note that for $p$ and $q$ in $\mathbb{R}^2_{\infty}$, if the points are in a diagonal position (i.e. there is $t\in \mathbb{R}$ such that $p=q+t\cdot(1,1)$ or $p=q+t\cdot(1,-1)$), then there is only one geodesic between $p$ and $q$. Of course, this is a line segment. If the points are not in a diagonal position, then there are infinitely many geodesics between these points. Likewise, for $p$ and $q$ in $\mathbb{R}^3_{\infty}$, if the points are in a cubic diagonal position (i.e. there is $t\in \mathbb{R}$ such that $p=q+t\cdot(1,1,1)$, $p=q+t\cdot(1,,1,-1)$, $p=q+t\cdot(1,,-1,1)$ or $p=q+t\cdot(-1,,1,1)$), then there is only one geodesic (which is, still, a line segment) between $p$ and $q$. If the points are not in a cubic diagonal position, then there are infinitely many geodesics between these points.

\begin{figure}[h]
\begin{center}
\begin{pspicture*}(-1.5,-1)(3,3.5)
\pspolygon*[linecolor=gray](0.5,1)(2.5,3)(2.5,-1)
\psline(0.5,1)(2.5,3) \psline(0.5,1)(2.5,-1)
\pspolygon*[linecolor=lightgray](0.5,1)(-1.5,-1)(2.5,-1)
\psline(0.5,1)(2.5,-1) \psline(0.5,1)(-1.5,-1)
\psline{->}(-1.5,0)(2.5,0) \uput[r](2.5,0){$x$}
\psline{->}(0,-1.5)(0,3) \uput[u](0,3){$y$} \psdot(0.5,1)
\uput[d](0.5,1){$p$} \uput[d](2,1.3){$S_1^+(p)$}
\uput[d](0.75,0){$S_2^-(p)$}
\psdot(2,-0.5) \uput[d](2,-0.5){$q$}
\end{pspicture*}
\caption{The points $p$ and $q$ are in a diagonal position.}
\label{fd4}
\end{center}
\end{figure}
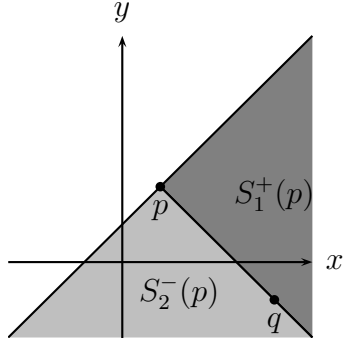

Note that $p$ and $q$ are in a diagonal position in $\mathbb{R}^2_{\infty}$ if and only if $q\in S_1^{\varepsilon}(p)\cap S_2^{\delta}(p)$ where $\varepsilon$ and $\delta$ are plus or minus. Likewise, $p$ and $q$ are in a cubic diagonal position in $\mathbb{R}^3_{\infty}$ if and only if $q\in S_1^{\varepsilon}(p)\cap S_2^{\delta}(p)\cap S_3^{\gamma}(p)$ where $\varepsilon$, $\delta$ and $\gamma$ are plus or minus (see Figure~\ref{fd4} and Figure~\ref{fd5}).

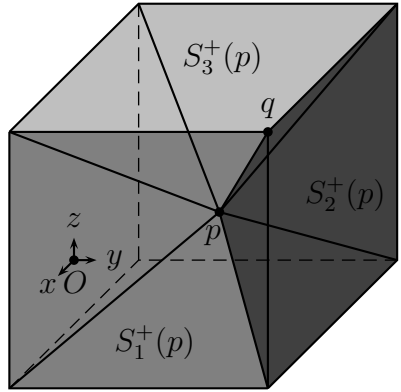
\begin{figure}[h!]
\begin{center}
\begin{pspicture*}(-2,-0.5)(8,6)
\psset{unit=0.85}
\pspolygon*[linecolor=lightgray](0,4)(3.25,2.75)(6,6)(2,6)(0,4)
\pspolygon*[linecolor=gray](0,4)(0,0)(4,0)(4,4)(0,4)
\pspolygon*[linecolor=darkgray](4,0)(6,2)(6,6)(4,4)(3.25,2.75)(4,0)
\psline(0,0)(3.25,2.75)(4,0)
\psline(3.25,2.75)(6,2)
\psline(0,4)(4,4)(4,0)
\psline(0,4)(2,6)(6,6)(6,2)(4,0)
\psline(6,6)(4,4)
\psline(4,0)(0,0)(0,4)
\psline[linewidth=0.5pt, linestyle=dashed](0,0)(2,2)(6,2)
\psline[linewidth=0.5pt, linestyle=dashed](2,2)(2,6)
\psdot(3.25,2.75)
\uput[d](3.15,2.75){$p$}
\psline(3.25,2.75)(2,6)
\psline(3.25,2.75)(6,6)
\psline(3.25,2.75)(0,4)
\psline(3.25,2.75)(4,4)
\uput[u](3.3,4.7){$S_3^{+}(p)$}
\psdot(1,2)
\uput[d](1.02,2.03){$O$}
\psline{->}(1,2)(0.75,1.75)
\uput[d](0.6,1.9){$x$}
\psline{->}(1,2)(1.35,2)
\uput[r](1.3,2){$y$}
\psline{->}(1,2)(1,2.35)
\uput[u](1,2.3){$z$}
\psdot(4,4)
\uput[u](4,4){$q$}
\uput[d](5.2,3.5){$S_2^+(p)$}
\uput[u](2.25,0.25){$S_1^+(p)$}
\end{pspicture*}
\caption{The points $p$ and $q$ are in a cubic diagonal position.}
\label{fd5}
\end{center}
\end{figure}

Now, let us consider a geodesic between two points in $\mathbb{R}^3_{\infty}$ and one belongs to intersection of only two sectors of the other one. Theorem~\ref{d1} implies that the image of such a geodesic must belong to the plane where these sectors intersect. So every geodesic between such points must be a planar curve (see Figure~\ref{fd6}).

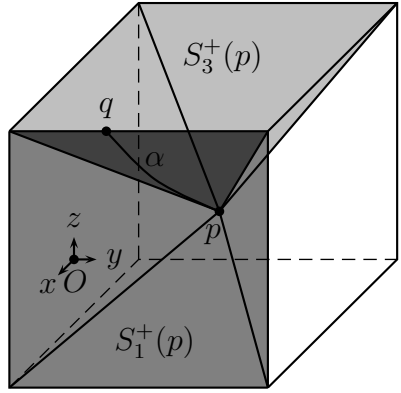
\begin{figure}[h!]
\begin{center}
\begin{pspicture*}(-2,-0.5)(8,6)
\psset{unit=0.85}
\pspolygon*[linecolor=lightgray](0,4)(3.25,2.75)(6,6)(2,6)(0,4)
\pspolygon*[linecolor=gray](0,4)(0,0)(4,0)(4,4)(0,4)
\pspolygon*[linecolor=darkgray](3.25,2.75)(4,4)(0,4)(3.25,2.75)
\psline(0,0)(3.25,2.75)(4,0)
\psline(0,4)(4,4)(4,0)
\psline(0,4)(2,6)(6,6)(6,2)(4,0)
\psline(6,6)(4,4)
\psline(4,0)(0,0)(0,4)
\psline[linewidth=0.5pt, linestyle=dashed](0,0)(2,2)(6,2)
\psline[linewidth=0.5pt, linestyle=dashed](2,2)(2,6)
\psdot(3.25,2.75)
\uput[d](3.15,2.75){$p$}
\psline(3.25,2.75)(2,6)
\psline(3.25,2.75)(6,6)
\psline(3.25,2.75)(0,4)
\psline(3.25,2.75)(4,4)
\uput[u](3.3,4.7){$S_3^{+}(p)$}
\psdot(1,2)
\uput[d](1.02,2.03){$O$}
\psline{->}(1,2)(0.75,1.75)
\uput[d](0.6,1.9){$x$}
\psline{->}(1,2)(1.35,2)
\uput[r](1.3,2){$y$}
\psline{->}(1,2)(1,2.35)
\uput[u](1,2.3){$z$}
\psdot(1.5,4)
\uput[u](1.5,4){$q$}
\uput[u](2.25,0.25){$S_1^+(p)$}
\pscurve(3.25,2.75)(2.25,3.25)(1.5,4)
\uput[u](2.25,3.25){$\alpha$}
\end{pspicture*}
\caption{A geodesic $\alpha$ between the points $p$ and $q$.}
\label{fd6}
\end{center}
\end{figure}

\section{Planes in $\mathbb{R}_{\infty}^3$}

Let $(X,d)$ be a metric space, $x$ and $y$ be any points in $X$. Then denote the number of geodesics between $x$ and $y$ by $\tau(x,y)$. For example, let $p$ and $q$ be in $\mathbb{R}_{\infty}^3$, then $\tau(p,q)=1$ if $p$ and $q$ are in a cubic diagonal position; otherwise $\tau(p,q)=\infty$.

Let $(X,d)$ be a metric space, $x\in X$ and $\varepsilon\in\mathbb{R}^+$. We define
\[
\nu(x,\varepsilon)=|\{y\in S_{\varepsilon}(x)\,|\ \tau(x,y)=1\}|
\]
where $S_{\varepsilon}(x)$ is the boundary of the disc of $x$ with radius $\varepsilon$: $$S_{\varepsilon}(x)=\{y\in X\,|\ d(x,y)=\varepsilon\}\,.$$
One can easily prove the following two propositions:

\begin{proposition}
Let $(X,d_X)$, $(Y,d_Y)$ be two metric space and $f:X\rightarrow Y$ be an isometry. Then we have
\[
\tau(x_1,x_2)=\tau(f(x_1),f(x_2))
\]
for all $x_1,x_2\in X$.
\end{proposition}

\begin{proposition}
\label{d2}
Let $(X,d_X)$, $(Y,d_Y)$ be two metric space and $f:X\rightarrow Y$ be an isometry. Then we have
\[
\nu(x,\varepsilon)=\nu(f(x),\varepsilon)
\]
for all $x\in X$ and $\varepsilon\in\mathbb{R}^+$.
\end{proposition}

\begin{figure}[h!]
\begin{center}
\begin{pspicture*}(-1.25,-0.5)(3,3)
\pspolygon*[linecolor=lightgray](-0.75,-0.25)(1.75,-0.25)(1.75,2.25)(-0.75,2.25)(-0.75,-0.25)
\pspolygon(-0.75,-0.25)(1.75,-0.25)(1.75,2.25)(-0.75,2.25)(-0.75,-0.25)
\psline{->}(-1.25,0)(2.5,0) \uput[r](2.5,0){$x$}
\psline{->}(0,-0.5)(0,2.75) \uput[r](0,2.75){$y$} \psdot(0.5,1)
\uput[d](0.5,1){$p$}
\psdots(-0.75,-0.25)(1.75,-0.25)(1.75,2.25)(-0.75,2.25)
\end{pspicture*}
\caption{Four points on the boundary of the $\varepsilon-$disc of the point $p$ which are connected to $p$ by only one geodesic.}
\label{fd7}
\end{center}
\end{figure}
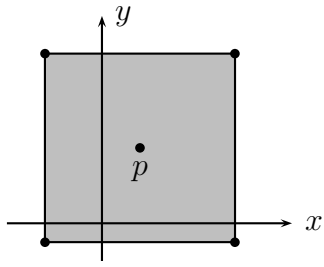

In the plane $\mathbb{R}_{\infty}^2$, for any point $p$ in it and for any positive number $\varepsilon$, $\nu(p,\varepsilon)=4$. These four points are obviously the vertices of the boundary of the $\varepsilon-$disc of the point $p$ which is a square (see Figure~\ref{fd7}).

In $\mathbb{R}_{\infty}^3$, for any point $p$ in it and for any positive number $\varepsilon$, $\nu(p,\varepsilon)=8$. These eight points are obviously the vertices of the boundary of the $\varepsilon-$disc of the point $p$ which is a cube (see Figure~\ref{fd8}).

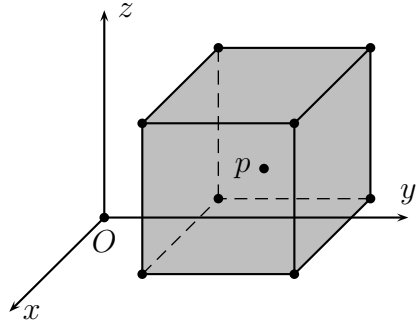
\begin{figure}[h!]
\begin{center}
\begin{pspicture*}(-2,-1.25)(4,3.75)
\pspolygon*[linecolor=lightgray](0,0)(2,0)(3,1)(3,3)(1,3)(0,2)(0,0)
\psline(0,2)(2,2)(2,0)
\psline(0,2)(1,3)(3,3)(3,1)(2,0)
\psline(3,3)(2,2)
\psline(2,0)(0,0)(0,2)
\psline[linewidth=0.5pt, linestyle=dashed](0,0)(1,1)(3,1)
\psline[linewidth=0.5pt, linestyle=dashed](1,1)(1,3)
\psdot(-0.5,0.75)
\uput[d](-0.5,0.75){$O$}
\psline{->}(-0.5,0.75)(-1.75,-0.5)
\uput[r](-1.75,-0.5){$x$}
\psline{->}(-0.5,0.75)(3.5,0.75)
\uput[u](3.5,0.75){$y$}
\psline{->}(-0.5,0.75)(-0.5,3.5)
\uput[r](-0.5,3.5){$z$}
\psdots(0,2)(1,3)(3,3)(3,1)(2,0)(2,2)(0,0)(1,1)(1.6,1.4)
\uput[l](1.6,1.4){$p$}
\end{pspicture*}
\caption{Eight points on the boundary of the $\varepsilon-$disc of the point $p$ which are connected to $p$ by only one geodesic.}
\label{fd8}
\end{center}
\end{figure}

Now we can ask a little more difficult question: Let $p$ be an arbitrary element in the plane $\{(x,y,z)\in\mathbb{R}_{\infty}^3\,|\,x+y+z=0\}\subseteq\mathbb{R}_{\infty}^3$ with induced maximum metric and $\varepsilon$ be any positive real number, then what is the number $\nu(p,\varepsilon)$? Maybe we must ask primarily what the boundary of the $\varepsilon-$disc of $p$ is. Of course, it is the intersection of the plane and the $\varepsilon-$cube of $p$ (i.e. the boundary of the $\varepsilon-$disc of $p$ in $\mathbb{R}_{\infty}^3$). It is surprisingly a regular hexagon (see Figure~\ref{fd9}).

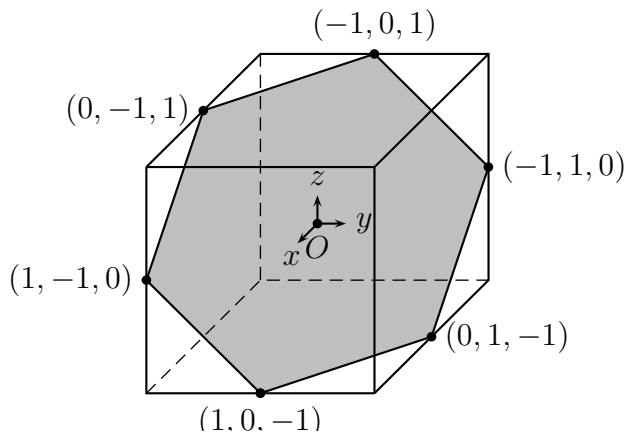
\begin{figure}[h!]
\begin{center}
\begin{pspicture*}(-2,-0.5)(8,6.5)
\psset{unit=0.75}
\pspolygon*[linecolor=lightgray](2,0)(5,1)(6,4)(4,6)(1,5)(0,2)(2,0)
\psline(2,0)(5,1)(6,4)(4,6)(1,5)(0,2)(2,0)
\psdots(2,0)(5,1)(6,4)(4,6)(1,5)(0,2)(2,0)
\psline(0,4)(4,4)(4,0)
\psline(0,4)(2,6)(6,6)(6,2)(4,0)
\psline(6,6)(4,4)
\psline(4,0)(0,0)(0,4)
\psline[linewidth=0.5pt, linestyle=dashed](0,0)(2,2)(6,2)
\psline[linewidth=0.5pt, linestyle=dashed](2,2)(2,6)
\psdot(3,3)
\uput[d](3,3){$O$}
\psline{->}(3,3)(2.65,2.65)
\uput[d](2.55,2.75){$x$}
\psline{->}(3,3)(3.5,3)
\uput[r](3.45,3){$y$}
\psline{->}(3,3)(3,3.5)
\uput[u](3,3.45){$z$}
\uput[d](2,0){$(1,0,-1)$}
\uput[r](5,1){$(0,1,-1)$}
\uput[r](6,4){$(-1,1,0)$}
\uput[u](4,6){$(-1,0,1)$}
\uput[l](1,5){$(0,-1,1)$}
\uput[l](0,2){$(1,-1,0)$}
\end{pspicture*}
\caption{Disc of the origin with radius $1$ in the plane $x+y+z=0$ is a regular hexagon.}
 \label{fd9}
\end{center}
\end{figure}
Obviously, the number $\nu(p,\varepsilon)$ is independent from $p$ and $\varepsilon$. So we can take $p$ as the origin and $\varepsilon=1$. If $q$ is a vertex on the hexagon, $q$ belongs to two sectors of the origin; therefore, any geodesic between the points origin and $q$ must be in the plane of intersection of these two sectors. (Note that all these intersection planes are $x=\pm y$, $x=\pm z$ and $y=\pm z$). Since the intersection of the former plane and the latter plane $x+y+z=0$ is the line passing trough the points $q$ and the origin, there is only one geodesic between these points in the plane $x+y+z=0$ which is the line segment. If $q$ is a point on the hexagon and it is not the vertex, then $q$ is contained by only one sector of the origin; thus, there are infinitely many geodesics between the points $q$ and the origin in the plane $x+y+z=0$. Hence, we have $\nu(p,\varepsilon)=6$ for all points $p$ in the plane $x+y+z=0$ and positive real numbers $\varepsilon$. Then, Proposition~\ref{d2} implies the following corollary:

\begin{corollary}
The plane $x+y+z=0$ in the $\mathbb{R}_{\infty}^3$ is not isometric to the plane $\mathbb{R}_{\infty}^2$.
\end{corollary}

Note that the plane $ax+by+cz=d$ isometric to the plane $ax+by+cz=0$ (to see this, consider the map $(x,y,z)\mapsto(x,y,z-\frac{d}{c})$); therefore, in order to classify all planes up to isometry in $\mathbb{R}_{\infty}^3$, it is enough to deal with the planes passing through the origin.

\begin{theorem}
\label{d3}
 The plane $ax+by+cz=0$ in $\mathbb{R}_{\infty}^3$ is not isometric to the plane $\mathbb{R}_{\infty}^2$ if and only if the number $|a|$, $|b|$ and $|c|$ are the edges of a non-degenerate triangle i.e. the inequalities $|a|,|b|,|c|\neq0$, $|a|+|b|>|c|$, $|a|+|c|>|b|$ and $|b|+|c|>|a|$ hold.
\end{theorem}

\begin{proof}
($\Rightarrow$)

Suppose that the numbers $|a|$, $|b|$ and $|c|$ are not the edges of a non-degenerate triangle. If two of these numbers are equal to zero, then the plane is the $xy-$plane, the $xz-$plane or the $yz-$plane and they are obviously isometric to the plane $\mathbb{R}_{\infty}^2$. Now, consider the case where one of these numbers is equal to zero. Without loss of generality, we may assume that $c=0$ and $|a|\geq |b|$. Then, all points on the plane are in the form of $(x,\frac{-a}{b}x,z)$ and the mapping $(x,\frac{-a}{b}x,z)\mapsto(\frac{-a}{b}x,z)$ is an isometry from the plane to $\mathbb{R}_{\infty}^2$ because
\[
|x_1-x_2|\leq \left|\frac{a}{b}\right|\cdot|x_1-x_2|
\]
for all $x_1,x_2\in \mathbb{R}$.

Now, let $a\neq0$, $b\neq0$ and $c\neq0$, but, for example, $|a|+|b|\leq|c|$. Then, all points on the plane are in the form of $(x,y,-\frac{ax+by}{c})$ and the mapping $(x,y,-\frac{ax+by}{c})\mapsto(x,y)$ is an isometry from the plane to $\mathbb{R}_{\infty}^2$ because

\begin{eqnarray*}
\left|\frac{ax_1+by_1}{c}-\frac{ax_2+by_2}{c}\right|&=&\left|\frac{a(x_1-x_2)+b(y_1-y_2)}{c}\right|\\
&\leq&\left|\frac{a}{c}\right|\cdot|x_1-x_2|+\left|\frac{b}{c}\right|\cdot|y_1-y_2|\\
&\leq&\left(\left|\frac{a}{c}\right|+\left|\frac{b}{c}\right|\right)\cdot\max\{\,|x_1-x_2|,\,|y_1-y_2|\,\}\\
&\leq&\max\{\,|x_1-x_2|,\,|y_1-y_2|\,\}\
\end{eqnarray*}
for all $x_1,x_2,y_1,y_2\in\mathbb{R}$.

($\Leftarrow$)

Let the number $|a|$, $|b|$ and $|c|$ be the edges of a non-degenerate triangle. Then, each face of the $1-$cube centered at the origin (i.e. $S_1(O)$ in $\mathbb{R}_{\infty}^3$) intersects the plane $ax+by+cz=0$ along a line segment. The reason for this is the following: For example, if we try to solve the equations $ax+by+cz=0$ and $x=1$ together, we can find infinitely many solutions for $y,z\in[-1,1]$ when $|b|+|c|>|a|$. This implies that the intersection of the plane and the square $x=1$ and $y,z\in[-1,1]$ is a line segment. As a result, intersection of the $1-$cube centered at the origin and the plane $ax+by+cz=0$ is a hexagon which is the boundary of a $1-$disc at the origin in the plane; hence, the number $\nu(O,1)=6$ and it is not isometric to the $\mathbb{R}_{\infty}^2$.
\end{proof}

\begin{remark}
Notice that Theorem \ref{d3} actually can be obtained using L. Nachbin's Theorem (Theorem 3 in \cite{nac}) and A. Moezzi's or M. Pavon's Theorem (Corollary 1.114 in \cite{moe} or Theorem 2.2 in \cite{pav}). But here we have given relatively more elementary proof for it.
\end{remark}

Note that the intersection of a cube and any plane passing through the center of the cube is a tetragon or a hexagon. Actually, Theorem~\ref{d3} says that if the intersection is a tetragon, then the plane is isometric to $\mathbb{R}_{\infty}^2$ and if the intersection is a hexagon, then the plane is not isometric to $\mathbb{R}_{\infty}^2$, so all planes which have the tetragonal intersection with the $1-$cube centered at the origin are isometric to each other. Is it true for the others? Or, are all planes in $\mathbb{R}_{\infty}^3$ which are not isometric to $\mathbb{R}_{\infty}^2$ isometric to each other? Note that every tetragon which is the intersection of a cube and a plane passing through the center of the cube has equal side lengths (see Figure~\ref{fd10}).
\begin{figure}[h!]
\begin{center}
\begin{pspicture*}(-2,-0.5)(8,6.5)
\psset{unit=0.75}
\pspolygon*[linecolor=lightgray](2,5.25)(0,1.25)(4,0.75)(6,4.75)(2,5.25)
\psline(2,5.25)(0,1.25)(4,0.75)(6,4.75)(2,5.25)
\psline(0,4)(4,4)(4,0)
\psline(0,4)(2,6)(6,6)(6,2)(4,0)
\psline(6,6)(4,4)
\psline(4,0)(0,0)(0,4)
\psline[linewidth=0.5pt, linestyle=dashed](0,0)(2,2)(6,2)
\psline[linewidth=0.5pt, linestyle=dashed](2,2)(2,6)
\psdot(3,3)
\uput[d](3,3){$O$}
\uput[d](2,0){$2$}
\uput[d](1.25,3.4){$2$}
\uput[d](2,1){$2$}
\uput[d](5.15,2.75){$2$}
\uput[d](3.8,5.1){$2$}
\end{pspicture*}
\caption{All edges of the tetragon above have same length which is equal to length of any edge of the cube in $\mathbb{R}_{\infty}^3$.}
 \label{fd10}
\end{center}
\end{figure}
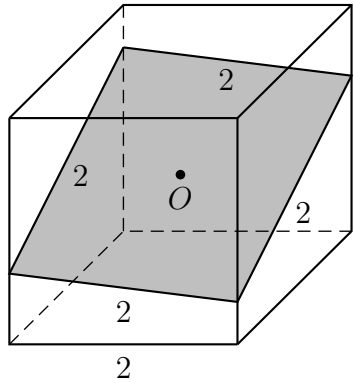
But the situation of the planes which have hexagonal intersection is different. For example, $x+y+z=0$ and $2x+2y+3z=0$ have different hexagons (one is regular, the other one is not) (see Figure~\ref{fd11}).

\begin{figure}[h!]
\begin{center}
\begin{pspicture*}(-2,-1)(18,6.5)
\psset{unit=0.75}
\pspolygon*[linecolor=lightgray](2,0)(5,1)(6,4)(4,6)(1,5)(0,2)(2,0)
\psline(2,0)(5,1)(6,4)(4,6)(1,5)(0,2)(2,0)
\psdots(2,0)(5,1)(6,4)(4,6)(1,5)(0,2)(2,0)
\uput[l](3.6,0.75){$1$}
\uput[l](5.5,2.4){$1$}
\uput[d](5.15,4.94){$1$}
\uput[r](2.4,5.2){$1$}
\uput[r](0.36,3.4){$1$}
\uput[r](1.1,0.86){$1$}
\psline(0,4)(4,4)(4,0)
\psline(0,4)(2,6)(6,6)(6,2)(4,0)
\psline(6,6)(4,4)
\psline(4,0)(0,0)(0,4)
\psline[linewidth=0.5pt, linestyle=dashed](0,0)(2,2)(6,2)
\psline[linewidth=0.5pt, linestyle=dashed](2,2)(2,6)
\psdot(3,3)
\uput[d](3,3){$O$}
\psline{->}(3,3)(2.65,2.65)
\uput[d](2.55,2.75){$x$}
\psline{->}(3,3)(3.5,3)
\uput[r](3.45,3){$y$}
\psline{->}(3,3)(3,3.5)
\uput[u](3,3.45){$z$}
\uput[d](2,0){$(1,0,-1)$}
\uput[r](5,1){$(0,1,-1)$}
\uput[r](6,4){$(-1,1,0)$}
\uput[u](4,6){$(-1,0,1)$}
\uput[l](1,5){$(0,-1,1)$}
\uput[l](0,2){$(1,-1,0)$}

\pspolygon*[linecolor=lightgray](14,0)(15.5,0.5)(17,4)(14,6)(12.5,5.5)(11,2)(14,0)
\psline(14,0)(15.5,0.5)(17,4)(14,6)(12.5,5.5)(11,2)(14,0)
\psdots(14,0)(15.5,0.5)(17,4)(14,6)(12.5,5.5)(11,2)(14,0)
\uput[u](14.65,0.15){$\frac{1}{2}$}
\uput[l](16.4,2.5){$\frac{3}{2}$}
\uput[l](15.35,4.9){$\frac{3}{2}$}
\uput[r](13.05,5.3){$\frac{1}{2}$}
\uput[d](11.9,3.9){$\frac{3}{2}$}
\uput[r](12.7,1.1){$\frac{3}{2}$}
\psline(11,4)(15,4)(15,0)
\psline(11,4)(13,6)(17,6)(17,2)(15,0)
\psline(17,6)(15,4)
\psline(15,0)(11,0)(11,4)
\psline[linewidth=0.5pt, linestyle=dashed](11,0)(13,2)(17,2)
\psline[linewidth=0.5pt, linestyle=dashed](13,2)(13,6)
\psdot(14,3)
\uput[d](14,3){$O$}
\psline{->}(14,3)(13.65,2.65)
\uput[d](13.55,2.75){$x$}
\psline{->}(14,3)(14.5,3)
\uput[r](14.45,3){$y$}
\psline{->}(14,3)(14,3.5)
\uput[u](14,3.45){$z$}
\uput[d](14,0){$(1,\frac{1}{2},-1)$}
\uput[r](15.5,0.5){$(\frac{1}{2},1,-1)$}
\uput[r](17,4){$(-1,1,0)$}
\uput[u](14,6){$(-1,-\frac{1}{2},1)$}
\uput[l](12.5,5.5){$(-\frac{1}{2},-1,1)$}
\uput[l](11,2){$(1,-1,0)$}
\end{pspicture*}
\caption{Two different planes (the left is $x+y+z=0$ and the right is $2x+2y+3z=0$) and their two different hexagons.}
 \label{fd11}
\end{center}
\end{figure}
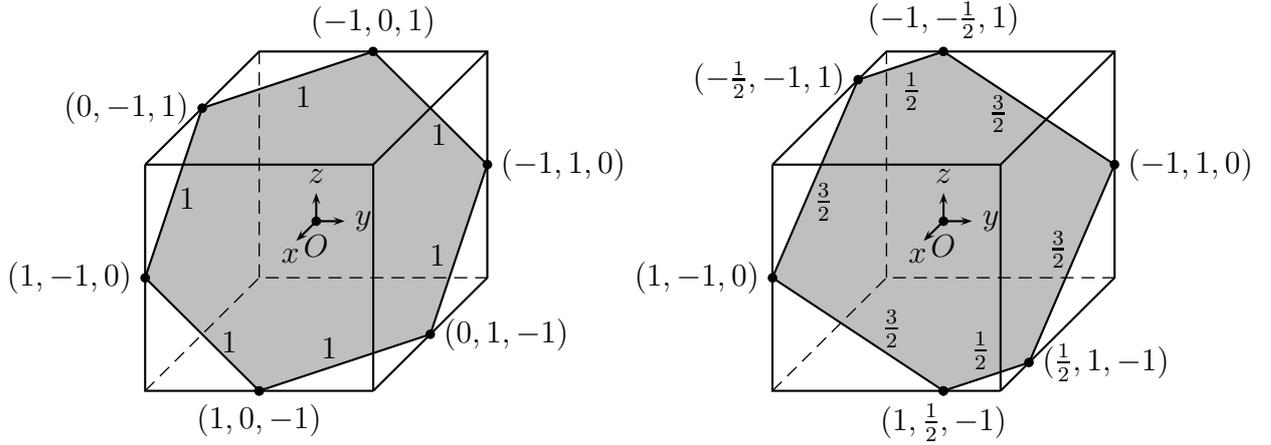

Note that the planes $x+y+z=0$ and $2x+2y+3z=0$ are not isometric. To see this, suppose $f$ is an isometry between these planes. Since a translation is an isometry from any plane to itself, we may assume that $f$ is fixing the origin. In this case, six vertex points on the hexagon in one plane must go isometrically to six vertex points on the hexagon in the other plane under $f$. However, it is obviously impossible. Actually, in order to determine isometry class of planes in $\mathbb{R}_{\infty}^3$, it is enough to determine isometry class of their hexagons ($1-$disc of the origin).

\begin{lemma}
\label{d4}
Let $ax+by+cz=0$ be any plane in $\mathbb{R}_{\infty}^3$. Then planes $\pm ax\pm by\pm cz=0$, $\pm ax\pm bz\pm cy=0$, $\pm ay\pm bx\pm cz=0$, $\pm az\pm bx\pm cy=0$, $\pm ay\pm bz\pm cx=0$ and $\pm az\pm by\pm cx=0$ are isometric to the $ax+by+cz=0$.
\end{lemma}

For an easy proof, look at the example $(x,y,z)\mapsto(z,-y,x)$. This map is obviously an isometry from the plane $ax+by+cz=0$ to the plane $az-by+cx=0$. Indeed, all the mappings $(x,y,z)\mapsto(\pm w_1,\pm w_2,\pm w_3)$, where $\{w_1,w_2,w_3\}=\{x,y,z\}$, give us the needed isometries and more. This is because, for instance, $(x,y,z)\mapsto(-x,-y,-z)$ is an isometry from $ax+by+cz=0$ to itself. Note that there are exactly 48 such maps and these are actually all the elements of isometry group of the cube. For example, $(x,y,z)\mapsto(x,-z,y)$ is $\frac{\pi}{2}$ counter clockwise rotation around the axis $x$ and $(x,y,z)\mapsto(x,-y,z)$ is the reflection with respect to the $xz-$plane. Denote this group by $G$ and let $X$ be the set of all planes passing through the origin in $\mathbb{R}_{\infty}^3$. Actually, $X$ can be thought as the set of all hexagons (on the $1-$cube of $\mathbb{R}_{\infty}^3$) passing through the origin. Then, $G$ acts on the $X$: Let the element $(x,y,z)\mapsto(w_1,w_2,w_3)$ be denoted by $(w_1,w_2,w_3)$ and defined by
\[
(w_1,w_2,w_3)\cdot (ax+by+cz=0):=aw_1+bw_2+w_3=0
\]
where $\{w_1,w_2,w_3\}=\{\pm x,\pm y,\pm z\}$.

Thus, Lemma~\ref{d4} can be restated as follows: A plane passing through the origin (in $\mathbb{R}_{\infty}^3$) is isometric to every plane which belongs to its orbit (according to group action above).

\begin{figure}[h!]
\begin{center}
\begin{pspicture*}(-2,-0.5)(8,6.5)
\psset{unit=0.75}
\psline(0,4)(4,4)(4,0)
\psline(0,4)(2,6)(6,6)(6,2)(4,0)
\psline(6,6)(4,4)
\psline(4,0)(0,0)(0,4)
\psline(0,3.25)(1.5,4)
\psline(0,2.5)(0.75,4)
\psline(0,1.5)(0.75,0)
\psline(0,0.75)(1.5,0)
\psline(3.25,0)(4,1.5)
\psline(2.5,0)(4,0.75)
\psline(4,2.5)(3.25,4)
\psline(4,3.25)(2.5,4)
\psline(1.5,4)(0.4,4.4)
\psline(0.75,4)(0.8,4.8)
\psline(1.6,5.6)(3.5,6)
\psline(1.2,5.2)(2.75,6)
\psline(5.25,6)(5.2,5.2)
\psline(4.5,6)(5.6,5.6)
\psline(3.25,4)(4.8,4.8)
\psline(2.5,4)(4.4,4.4)
\psline(4,3.25)(4.8,4.8)
\psline(4,2.5)(4.4,4.4)
\psline(4,0.75)(4.8,0.8)
\psline(4,1.5)(4.4,0.4)
\psline(6,3.5)(5.6,1.6)
\psline(6,2.75)(5.2,1.2)
\psline(6,5.25)(5.2,5.2)
\psline(6,4.5)(5.6,5.6)
\psline[linewidth=0.5pt, linestyle=dashed](0,0)(2,2)(6,2)
\psline[linewidth=0.5pt, linestyle=dashed](2,2)(2,6)
\psdot(3,3)
\uput[d](3,3){$O$}
\end{pspicture*}
\caption{The edges of the $24$ isometric hexagons (of the $24$ isometric planes).}
 \label{fd12}
\end{center}
\end{figure}
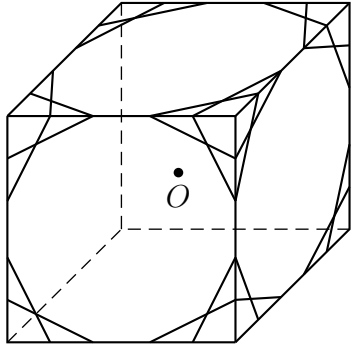
Let $ax+by+cz=0$ be an element of $X$ such that it is not isometric to the $\mathbb{R}_{\infty}^2$. Consider the case where $|a|,|b|$ and $|c|$ are distinct numbers. Then, the orbit of this plane has exactly $24$ elements because its stabilizer contains only two elements: Identity and $(-x,-y,-z)$. What are these $24$ isometric planes or their hexagons? Note that a hexagon can be determined by only one edge because the plane containing this edge and the origin is unique and the hexagon is intersection of this plane and the cube. An example of the edges of $24$ isometric planes is in Figure~\ref{fd12}.

\begin{figure}[h!]
\begin{center}
\begin{pspicture*}(-0.5,-0.5)(15,6.5)
\psset{unit=0.75}
\psline(0,4)(4,4)(4,0)
\psline(0,4)(2,6)(6,6)(6,2)(4,0)
\psline(6,6)(4,4)
\psline(4,0)(0,0)(0,4)
\psline(0,1)(3,4)
\psline(0,3)(3,0)
\psline(1,0)(4,3)
\psline(4,1)(1,4)
\psline(3,4)(1.5,5.5)
\psline(0.5,4.5)(5,6)
\psline(3,6)(4.5,4.5)
\psline(1,4)(5.5,5.5)
\psline(4,1)(5.5,5.5)
\psline(4,3)(5.5,1.5)
\psline(6,5)(4.5,0.5)
\psline(6,3)(4.5,4.5)
\psline[linewidth=0.5pt, linestyle=dashed](0,0)(2,2)(6,2)
\psline[linewidth=0.5pt, linestyle=dashed](2,2)(2,6)
\psdot(3,3)
\uput[d](3,3){$O$}

\psline(8,4)(12,4)(12,0)
\psline(8,4)(10,6)(14,6)(14,2)(12,0)
\psline(14,6)(12,4)
\psline(12,0)(8,0)(8,4)
\psline(8,2.5)(9.5,4)
\psline(8,1.5)(9.5,0)
\psline(10.5,0)(12,1.5)
\psline(12,2.5)(10.5,4)
\psline(9.5,4)(8.8,4.8)
\psline(9.2,5.2)(11.5,6)
\psline(12.5,6)(13.2,5.2)
\psline(10.5,4)(12.8,4.8)
\psline(12,2.5)(12.8,4.8)
\psline(12,1.5)(12.8,0.8)
\psline(14,3.5)(13.2,1.2)
\psline(14,4.5)(13.2,5.2)
\psline[linewidth=0.5pt, linestyle=dashed](8,0)(10,2)(14,2)
\psline[linewidth=0.5pt, linestyle=dashed](10,2)(10,6)
\psdot(11,3)
\uput[d](11,3){$O$}
\end{pspicture*}
\caption{Two examples of the edges of the $12$ isometric hexagons (of the $12$ isometric planes).}
 \label{fd13}
\end{center}
\end{figure}
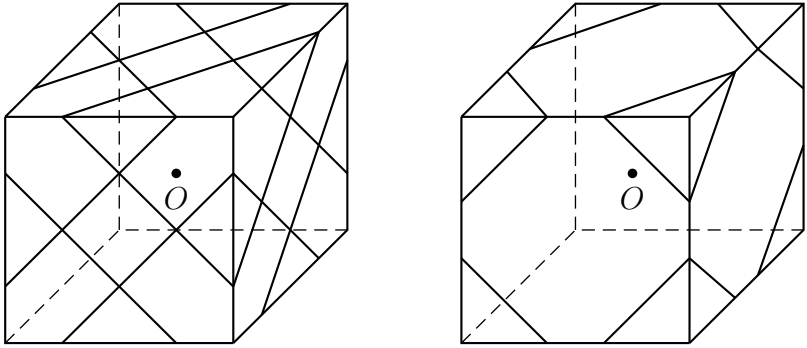

Now, consider the case only two of the numbers $|a|,|b|$ and $|c|$ are equal. Say $a=b$. Then, the orbit of this plane has exactly $12$ elements because its stabilizer contains only four elements: Identity, $(-x,-y,-z)$, $(y,x,z)$ and $(-y,-x,-z)$. Two examples of the edges of $12$ isometric planes are given in Figure~\ref{fd13}.

\begin{figure}[h!]
\begin{center}
\begin{pspicture*}(-2,-0.5)(8,6.5)
\psset{unit=0.75}
\psline(0,4)(4,4)(4,0)
\psline(0,4)(2,6)(6,6)(6,2)(4,0)
\psline(6,6)(4,4)
\psline(4,0)(0,0)(0,4)
\psline(0,2)(2,4)
\psline(0,2)(2,0)
\psline(2,0)(4,2)
\psline(4,2)(2,4)
\psline[linewidth=0.5pt, linestyle=dashed](0,0)(2,2)(6,2)
\psline[linewidth=0.5pt, linestyle=dashed](2,2)(2,6)
\psdot(2.5,2.5)
\uput[d](2.5,2.5){$O$}
\end{pspicture*}
\caption{The edges of the $4$ isometric hexagons (of the $4$ isometric planes).}
 \label{fd14}
\end{center}
\end{figure}
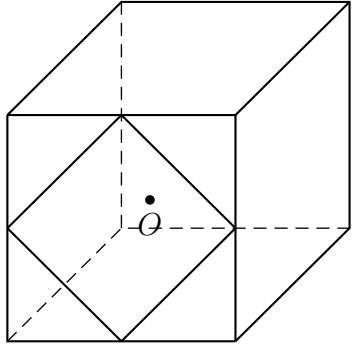

The last case is $|a|=|b|=|c|$; that is, the plane is $x+y+z=0$. Then, the orbit of this plane has exactly $4$ elements. An example of the edges of $4$ isometric planes is in Figure~\ref{fd14}.

\begin{lemma}
\label{d5}
Let $|a|,|b|$ and $|c|$ be the edges of a non-degenerate triangle. Then the plane $ax+by+cz=0$ is isometric to only planes passing through the origin which belong to its orbit.
\end{lemma}

For the proof, assume not. This implies that there are two isometric hexagons such that one does not belong to the others orbit. But this is impossible (see Figure~\ref{fd15}).

\begin{figure}[h!]
\begin{center}
\begin{pspicture*}(-0.5,-0.5)(15,6.5)
\psset{unit=0.75}
\psline(0,4)(4,4)(4,0)
\psline(0,4)(2,6)(6,6)(6,2)(4,0)
\psline(6,6)(4,4)
\psline(4,0)(0,0)(0,4)
\psline(6,4)(5.5,5.5)(1,4)(0,2)
\psline[linecolor=red](6,4.5)(5.5,5.5)(1.75,4)(0,1.5)
\psline[linecolor=yellow](6,5)(5.5,5.5)(2.5,4)(0,1)
\psline[linecolor=green](6,5.5)(5.5,5.5)(3.25,4)(0,0.5)
\psline[linewidth=0.5pt, linestyle=dashed](0,0)(2,2)(6,2)
\psline[linewidth=0.5pt, linestyle=dashed](2,2)(2,6)
\psdot(3,3)
\uput[d](3,3){$O$}

\psline(8,4)(12,4)(12,0)
\psline(8,4)(10,6)(14,6)(14,2)(12,0)
\psline(14,6)(12,4)
\psline(12,0)(8,0)(8,4)
\psline(14,4)(13.5,5.5)(9,4)(8,2)
\psline[linecolor=red](14,3.5)(13,5)(9,4)(8,2.5)
\psline[linecolor=green](14,3)(12.5,4.5)(9,4)(8,3)
\psline[linewidth=0.5pt, linestyle=dashed](8,0)(10,2)(14,2)
\psline[linewidth=0.5pt, linestyle=dashed](10,2)(10,6)
\psdot(11,3)
\uput[d](11,3){$O$}
\end{pspicture*}
\caption{All the line segments on the $z=1$ plane are equal with respect to the maximum metric. But the line segments on the $x=1$ plane are not equal on the left and the line segments on the $y=1$ plane are not equal on the right.}
\label{fd15}
\end{center}
\end{figure}
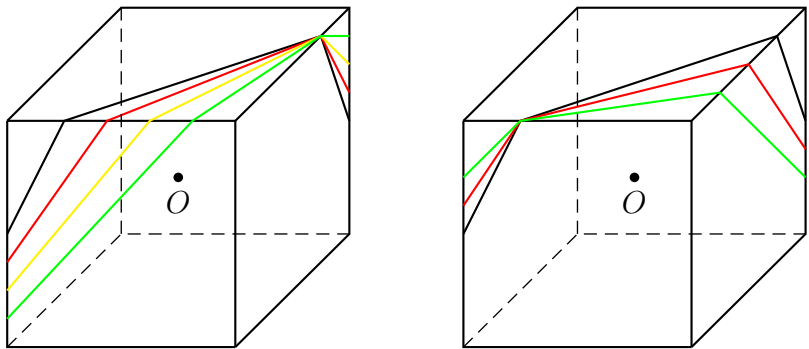

Note that Lemma~\ref{d5} says that the number of the isometry classes of planes passing through the origin in $\mathbb{R}_{\infty}^3$ can be identified by the number of the similarity classes of triangles in the Euclidean plane. If the numbers $|a|$, $|b|$ and $|c|$ are not lengths of edges of a non-degenerate triangle, then the plane $ax+by+cz=0$ is isometric to the $\mathbb{R}_{\infty}^2$, so such planes form one isometry class. If the numbers $|a|,|b|,|c|$ and $|a'|,|b'|,|c'|$ are lengths of edges of two non-degenerate triangles, then these triangles are similar if and only if the planes $ax+by+cz=0$ and $a'x+b'y+c'z=0$ are isometric. As a result, we get the following corollary:

\begin{corollary}
Let $ax+by+cz=d$ be a plane in $\mathbb{R}_{\infty}^3$.
\begin{itemize}
\item[i)]
If the numbers $|a|,|b|$ and $|c|$ are not lengths of edges of a non-degenerate triangle, then this plane is isometric to the $\mathbb{R}_{\infty}^2$, so such planes are isometric to each other.
\item[ii)]
If the numbers $|a|,|b|$ and $|c|$ are lengths of edges of a non-degenerate triangle, then this plane is isometric to only planes whose equations can be written as $aw_1+bw_2+cw_3=D$ where $\{w_1,w_2,w_3\}=\{\pm x,\pm y,\pm z\}$ and $D$ is any real number.
\end{itemize}
\end{corollary}

Email: kompaktuzay@gmail.com

Anadolu University, Faculty of Science, Department of Mathematics, Eskisehir/ TURKEY


\begin{thebibliography}{1}

\bibitem{bri} M.R. Bridson, A. Haefliger, \textit{Metric Spaces of Non-Positive Curvature}.  Grundlehren der mathematischen Wissenschaften, Springer-Verlag, Berlin, 1999.

\bibitem{bur} D. Burago, Y. Burago, S. Ivanov. A Course in Metric Geometry, Graduate Studies in Mathematics. American Mathematical Society, USA, 2001.

\bibitem{kil} M. Kilic, S. Kocak. Tight Span of Subsets of The Plane With The Maximum Metric, Advances in Mathematics 301, 693-710, 2016.

\bibitem{moe} A. Moezzi, The Injective Hull of Hyperbolic Groups, Dissertation ETH Zurich, No. 18860, 2010.

\bibitem{nac} L. Nachbin, A Theorem of The Hahn-Banach Type For Linear Transformations, Trans. Amer. Math. Soc. 68, 28-46 (1950).

\bibitem{pap} Papadopoulos A.\emph{} Metric Spaces, Convexity and Nonpositive Curvature. Irma Lectures in Mathematics and Theoretical Physics, European Mathematical Society, Germany, 2005.

\bibitem{pav} M. Pavon, Injective Convex Polyhedra. Discrete, Computational Geometry (2016). doi:10.1007/s00454-016-9810-6

\end{thebibliography}
\end{document}